\documentclass[amstex,10pt,reqno]{amsart}
\usepackage{amsmath,amsfonts,amssymb,amsthm,enumerate,hyperref,multicol, graphicx, wrapfig, pdfsync,array, caption, subcaption, color}
\usepackage[dvipsnames]{xcolor}
\newtheorem{lemma}{Lemma}
\newtheorem{thm}{Theorem}
\newtheorem{definition}{Definition}
\newtheorem{example}{Example}
\newtheorem{corollary}{Corollary}

\numberwithin{equation}{section}
\begin{document}

\leftline{ \scriptsize \it}
\title[]
{BETTER APPROXIMATION OF FUNCTION BY $\alpha-$BERNSTEIN-P\u{A}LT\u{A}NEA OPERATORS}
\maketitle

\begin{center}
{\bf  Jaspreet Kaur$^1$ and Meenu Goyal$^1$}
\vskip0.2in
$^{1}$School of Mathematics\\
Thapar Institute of Engineering and Technology, Patiala\\
Patiala-147004, India\\
\vskip0.2in
jazzbagri3@gmail.com and meenu\_rani@thapar.edu
\end{center}
%
\begin{abstract}
In this paper, we present a new type of $\alpha-$Bernstein-P\u{a}lt\u{a}nea operators having a better order of approximation than itself. We establish some approximation results concerning the rate of convergence, error estimation and asymptotic formulas for the new modifications. Also, the theoretical results are verified by using MAPLE algorithms.
\end{abstract}
\section{\bf{{Introduction}}}
Weierstrass's approximation theorem is one of the most important theorem in approximation theory. In 1912, S. N. Bernstein \cite{BER} proved this theorem by using the positive and linear operators, which are known as Bernstein operators. For $f\in C[0,1],$ they are defined as
$$\mathcal{B}_n(f,x)=\sum_{k=0}^np_{n,k}(x)f\left(\frac{k}{n}\right), $$
where $p_{n,k}(x)= {n \choose k}  x^k  (1-x)^{n-k}.$ He proved that the operators $\mathcal{B}_{n}(f,x)$ converge uniformly to $f$ on $C[0,1].$ Although these operators are known by their simplicity and have some remarkable properties of approximations, yet these positive linear operators have slow rate of convergence. In order to make it more attractive, in a computational point of view, several modifications and improvements have been investigated by many authors (see \cite{PLB, DT, GA, CAM}).

In \cite{CHE}, Chen et al. gave a new family of generalized Bernstein operators depending upon a non-negative real parameter $0 \leq \alpha \leq 1,$ which is given as follows:
\begin{eqnarray}\label{m2}
\mathcal{T}_{n}^{\alpha}(f,x)=\sum_{k=0}^n p_{n,k}^{\alpha}(x)f\left(\frac{k}{n}\right),
\end{eqnarray}
where $f(x)$ defined on $[0,1]$ and $n\in \mathbb{N}.$
Also, the $\alpha-$Bernstein polynomial $p_{n,k}^{\alpha}(x)$ of degree $n$ is defined by $p_{1,0}^{\alpha}(x)=1-x, \quad p_{1,1}^{\alpha}(x)=x$ and
$$p_{n,k}^{\alpha}(x)=\left[{n-2\choose k}x(1-\alpha)+{n-2 \choose k-2}(1-x)(1-\alpha)+{n \choose k}\alpha x(1-x)\right]x^{k-1}(1-x)^{n-k-1}, \quad n\geq2.$$
For $\alpha =1,$ it reduces to original Bernstein operators. These operators have certain elementary properties, which give significant contribution in uniform convergence of functions without depending on the parameter $\alpha.$\\
As we can see positive linear summation operators are useful for the convergence of only continuous functions. So, in 1967, Durrmeyer \cite{DU7} modified the Bernstein operators to approximate Lebesgue integrable functions on $[0,1].$ The Durrmeyer variant of Bernstein operators attracted attention of several authors (see \cite{MMD, ZD}), due to its usefulness. Recently, Kajla and Acar \cite{KA} introduced the Durrmeyer variant of the summation operators $(\ref{m2})$ and studied the rate of convergence and some approximation properties. Most recently, Kajla and Goyal \cite{KG} modified these Durrmeyer operators by using P\u{a}lt\u{a}nea basis function in an integral depending on a parameter $\rho>0$ as follows:
\begin{eqnarray}\label{OP}
\textit{Q}_{n,\rho}^{\alpha}(f,x)=\sum_{k=0}^n p_{n,k}^{\alpha} \int_0^1 \mu_{n,k}^\rho(t)f(t)dt,
\end{eqnarray}
where
$$\mu_{n,k}^\rho(t)=\frac{t^{k\rho}(1-t)^{(n-k)\rho}}{B(k\rho+1,(n-k)\rho+1)}$$
and $B(k\rho+1,(n-k)\rho+1)$ is Beta function. They have studied the approximation properties, asymptotic behaviour and the order of convergence of these operators.\\
Here, our main motive is to improve approximation behaviour and order of convergence for the operators (\ref{OP}). In \cite{KA7}, Khosravian-Arab et al. modified the well known Bernstein operators by using a new approach to improve the degree of approximation. Following this, Acu et al. \cite{DUR} have applied this approach on the Bernstein-Durrmeyer operators. In an another paper \cite{GTA}, same authors have put it on the Bernstein-Kantorovich operators too. Similarly, Kajla and Acar \cite{Kaj} have modified the $\alpha-$Bernstein summation operators. The inspiration of getting better approximation results for positive linear operators leads us to modify the $\alpha-$Bernstein-P\u{a}lt\u{a}nea operators that is defined in (\ref{OP}). In the present paper, we establish an approach to construct some modifications of these operators which have better convergence behaviour than the classical one.

Our work is organized as follows: In section 2, we define $\alpha-$Bernstein-P\u{a}lt\u{a}nea operators of first order which yields better results than the original operators. In section 3 and 4, we introduce $\alpha-$Bernstein-P\u{a}lt\u{a}nea operators of second order and third order respectively which possess better order of approximation than the operators (\ref{OP}). In section 5, we verify the theoretical results obtained in section 2-4 numerically using Maple algorithms.

\section{$\alpha-$Bernstein-P\u{a}lt\u{a}nea operators of first order}
In this section, we define $\alpha-$Bernstein-P\u{a}lt\u{a}nea operators of first order.
\begin{eqnarray}\label{a1}
\textit{J}_{n,\alpha,\rho}^{M,1}(f,x) &=& \sum_{k=0}^n p_{n,k,\alpha}^{M,1}(x)\int_ 0^1 \mu_{n,k}^\rho(t)f(t),  x\in[0,1]\\
p_{n,k,\alpha}^{M,1}(x) &=& a(x,n)p_{n-1,k,\alpha}(x)+a(1-x,n)p_{n-1,k-1,\alpha}(x), 0\leq k\leq n-1,
\end{eqnarray}
and $a(x,n)=a_1(n)x+a_0(n),$
where $a_0(n)$ and $a_1(n)$ are two unknown sequences, which can be determined to satisfy our purposes. For $a_0(n)=1$ and $a_1(n)=-1,$ it reduces to the operators (\ref{OP}). Now, we compute some preliminary results which will be useful to study the uniform convergence and asymptotic results. For this we assume $e_i= x^i, i=0,1,2,\dotsb.$
\begin{lemma}\label{l1}
The moments of the operators (\ref{a1}) are given by:
\begin{eqnarray*}
\textit{J}_{n,\alpha,\rho}^{M,1}(e_0;x) &=& (2a_0(n)+a_1(n));\\
\textit{J}_{n,\alpha,\rho}^{M,1}(e_1;x) &=& (2a_0(n)+a_1(n))x+\frac{1}{n\rho+2}\left[(1-2x)(a_0(n)(\rho+2)+a_1(n)(\rho+1))\right];\\
\textit{J}_{n,\alpha,\rho}^{M,1}(e_2;x) &=& (2a_0(n)+a_1(n))x^2+ \frac{1}{(n\rho+2)(n\rho+3)}\left[n\left\{\rho x(3-5x)(2a_0(n)+a_1(n))+{\rho}^2x(a_0(n)(4-6x)\right.\right.\\
&& +\left.a_1(n)(3-5x))\right\} +(-6x^2+2\alpha {\rho}^2x^2-2\alpha {\rho}^2x+2)(2a_0(n)+a_1(n))\\
&& +\left.\rho(\rho+3-6x)a_0(n)+\rho(2\rho x^2-6x-2\rho x+\rho+3)a_1(n)\right].
\end{eqnarray*}
\end{lemma}

\begin{lemma}\label{l2}
The central moments of the operators $\textit{J}_{n,\alpha,\rho}^{M,1}(f;x)$ are given as:
\begin{eqnarray*}
\textit{J}_{n,\alpha,\rho}^{M,1}(t-x;x) &=& \frac{1}{n\rho+2}\left[(1-2x)(a_0(n)(\rho+2)+a_1(n)(\rho+1))\right];\\
\textit{J}_{n,\alpha,\rho}^{M,1}((t-x)^2;x) &=& \frac{1}{(n\rho+2)(n\rho+3)} \left[x(1-x)\rho(1+\rho)(2a_0(n)+a_1(n))n\right.\\
&& -x(1-x)\left(4a_0(n)(3+3\rho+\alpha{\rho}^2)+2a_1(n)(3+6\rho+{\rho}^2(1+\alpha))\right)\\
&& \left.+ a_0(n)(4+3\rho+{\rho}^2)+a_1(n)(2+3\rho+{\rho}^2)\right];\\
\textit{J}_{n,\alpha,\rho}^{M,1}((t-x)^4;x) &=& \frac{1}{(n\rho+2)(n\rho+3)(n\rho+4)(n\rho+5)} \left[3{\rho}^2(1+{\rho})^2x^2(1-x)^2(2a_0(n)+a_1(n))n^2\right.\\
&& + n\left\{-12\alpha {\rho}^3(1+\rho)x^2(1-x)^2(2a_0(n)+a_1(n))\right.\\
&& + a_0(n)x(1-x)\rho[52+196\rho+56{\rho}^2+12{\rho}^3-x(1-x)(172+344\rho+200{\rho}^2+28{\rho}^3)]\\
&& + \left.a_1(n)x(1-x)\rho[26+61\rho+46{\rho}^2+11{\rho}^3-x(1-x)(86+224\rho+172{\rho}^2+34{\rho}^3)]\right\}\\
&& + \left\{\alpha x(1-x)\left[{\rho}^2(-70+240x(1-x))(2a_0(n)+a_1(n))\right.\right.\\
&& + 60x{\rho}^3(-1+4x(1-x))(3a_0(n)+a_1(n))\\
&& + \left.{\rho}^4(a_0(n)(-64+288x(1-x))+a_1(n)(-50+216x(1-x)))\right]\\
&& + a_0(n)(48+50\rho+35{\rho}^2+10{\rho}^3+{\rho}^4+x(-240-320\rho-120{\rho}^2+80{\rho}^3+48{\rho}^4)\\
&& +x^2(480+800\rho+120{\rho}^2-560{\rho}^3-288{\rho}^4)+x^3(-480-960\rho+960{\rho}^3+480{\rho}^4)\\
&& +x^4(240+480\rho-480{\rho}^3-240{\rho}^4))\\
&& + a_1(n)(24+50\rho+35{\rho}^2+10{\rho}^3+{\rho}^4+x(-120-320\rho-190{\rho}^2+20{\rho}^3+34{\rho}^4)\\
&& +x^2(240+800\rho+430{\rho}^2-260{\rho}^3-202{\rho}^4)+x^3(-240-960\rho-480{\rho}^2+480{\rho}^3+336{\rho}^4)\\
&& +\left.\left.x^4(120+480\rho+240{\rho}^2-240{\rho}^3-168{\rho}^4))\right\}\right].
\end{eqnarray*}
\end{lemma}
To obtain uniform convergence of the operators (\ref{a1}),  throughout this paper, we will assume for the sequences $a_i(n), i=0,1$
\begin{eqnarray}\label{eq1}
2a_0(n)+a_1(n)=1.
\end{eqnarray}
Depending on the choices of the sequences, we get two cases which are given by:\\
Case 1. Let
\begin{eqnarray}\label{eq2}
a_0(n)\geq 0, a_0(n)+a_1(n)\geq 0.
\end{eqnarray}
From this, we get $0\leq a_0(n)\leq1$ and $-1\leq a_1(n)\leq1.$ So, both sequences are bounded. Also, the operators (\ref{a1}) are positive for this case.\\
Case 2. Let
\begin{eqnarray}\label{eq3}
a_0(n)<0 \,\,\mbox{or}\,\, a_1(n)+a_0(n)<0.
\end{eqnarray}
If $a_0(n)<0,$ then $a_1(n)+a_0(n)>1$ and if $a_0(n)+a_1(n)<0,$ then $a_0(n)>1.$ In this case the operators (\ref{a1}) are not positive.\\
Firstly, we prove the basic convergence and asymptotic results for case 1.

\begin{thm}\label{t1}
Let $f\in C[0,1].$ If $a_0(n),a_1(n)$ satisfy the conditions (\ref{eq1}) and (\ref{eq2}), then
$$\lim_{n\rightarrow\infty} \textit{J}_{n,\alpha,\rho}^{M,1}(f,x)=f(x),$$
uniformly on [0,1].
\end{thm}
\begin{proof} Since for the conditions on $a_i(n), i=0,1$ the operators (\ref{a1}) are positive. So, by Korovkin theorem and Lemma \ref{l1}, we can find the uniform convergence of the operators.
\end{proof}

\begin{thm}\label{t2}
Let $a_i(n), i=0,1$ are convergent sequences satisfying the conditions (\ref{eq1})-(\ref{eq2}) and $l_i=\displaystyle\lim_{n \rightarrow \infty}a_i(n).$ If $f^{\prime\prime} \in C[0,1],$ then:
\begin{eqnarray*}
\lim_{n\rightarrow \infty}n(\textit{J}_{n,\alpha,\rho}^{M,1}(f;x)-f(x))=\frac{(1-2x)((\rho+2)l_0+(\rho+1)l_1)}{\rho}f^\prime(x)
+\frac{x(1-x)(1+\rho)(2l_0+l_1)}{2\rho}f^{\prime\prime}(x),
\end{eqnarray*}
uniformly on [0,1].
\end{thm}
\begin{proof} By the Taylor's formula, we have
$$f(t)=f(x)+(t-x)f^{\prime}(x)+\frac{1}{2}(t-x)^2 f^{\prime\prime}(x)+\theta(t,x)(t-x)^2.$$
Apply the operators $\textit{J}_{n,\alpha,\rho}^{M,1}(.,x)$ on Taylor's formula, we get:
\begin{eqnarray*}
\textit{J}_{n,\alpha,\rho}^{M,1}(f;x)-f(x)=\textit{J}_{n,\alpha,\rho}^{M,1}((t-x);x)f^{\prime}(x)
+\frac{1}{2}\textit{J}_{n,\alpha,\rho}^{M,1}((t-x)^2);x)f^{\prime\prime}(x)+\textit{J}_{n,\alpha,\rho}^{M,1}(\theta(t,x)(t-x)^2;x),
\end{eqnarray*}
where $\theta(t,x) \in C[0,1]$ and $\displaystyle\lim_{t\rightarrow x} \displaystyle\theta(t,x)=0.$\\
Using Cauchy-Schwarz inequality on the last term of the above equation, we obtain
\begin{eqnarray*}
\textit{J}_{n,\alpha,\rho}^{M,1}(\theta(t,x)(t-x)^2;x)\leq \sqrt{\textit{J}_{n,\alpha,\rho}^{M,1}({\theta}^2(t,x);x)} \sqrt{\textit{J}_{n,\alpha,\rho}^{M,1}((t-x)^4;x)}.
\end{eqnarray*}
Since ${\theta}^2(x,x)=0$ and ${\theta}^2(t,x) \in C[0,1],$ and using Theorem \ref{t1}, then we have $\displaystyle\lim_{n \rightarrow \infty} \textit{J}_{n,\alpha,\rho}^{M,1}({\theta}^2(t,x);x)=0$ uniformly on [0,1].\\

Hence, from Lemma \ref{l2}, we obtain
$$n\lim_{n\rightarrow \infty}\textit{J}_{n,\alpha,\rho}^{M,1}(\theta(t,x)(t-x)^2;x)=0.$$
By using the central moments from Lemma \ref{l2}, we get the required result.
\end{proof}

Now, we have the convergence and asymptotic results for the case 2:
\begin{thm}\label{t3}
Let $f\in C[0,1]$ and $a_i(n), i=0,1$ be a convergent sequences which satisfy the conditions (\ref{eq1}) and (\ref{eq3}). Then:
$$\lim_{n\rightarrow\infty} \textit{J}_{n,\alpha,\rho}^{M,1}(f,x)=f(x),$$
uniformly on [0,1].
\end{thm}
\begin{proof} Our operators can be written as:
\begin{eqnarray}\label{m1}
\textit{J}_{n,\alpha,\rho}^{M,1}(f,x)=\textit{K}_{n,\alpha,\rho}^{M,1}(f,x)-\textit{L}_{n,\alpha,\rho}^{M,1}(f,x)
\end{eqnarray}
where
\begin{eqnarray*}
\textit{K}_{n,\alpha,\rho}^{M,1}(f,x) &=& \sum_{k=0}^{n}[a_1(n)~x~ p_{n-1,k,\alpha}(x)+a_1(n)~ p_{n-1,k-1,\alpha}(x)]\int_{0}^1 \mu_{n,k}^\rho(t) f(t)dt,\\
\textit{L}_{n,\alpha,\rho}^{M,1}(f,x) &=& \sum_{k=0}^n [-a_0(n)~p_{n-1,k,\alpha}(x)+(a_1(n)x-a_0(n))~p_{n-1,k-1,\alpha}(x)]\int_{0}^1 \mu_{n,k}^\rho(t) f(t)dt.
\end{eqnarray*}
Since both the operators i.e. $\textit{K}_{n,\alpha,\rho}^{M,1}(f,x)$ and $\textit{L}_{n,\alpha,\rho}^{M,1}(f,x)$ are positive. So, we can apply extended Korovkin theorem (\cite{KA7}, see page 122) on it. The moments of these operators are given below:
\begin{eqnarray*}
\textit{K}_{n,\alpha,\rho}^{M,1}(e_0,x) &=& a_1(n)(1+x),\\
\textit{K}_{n,\alpha,\rho}^{M,1}(e_1,x) &=& a_1(n)(1+x)\left[\frac{(n-1)\rho x}{n\rho+2}\right]+\frac{a_1(n)x}{n\rho+2}+\frac{a_1(n)(\rho+1)}{n\rho+2},\\
\textit{K}_{n,\alpha,\rho}^{M,1}(e_2,x) &=& \frac{a_1(n)(1+x)}{(n\rho+2)(n\rho+3)} \left[{\rho}^2(n-1)^2 \left(x^2+\frac{n-1+2(1-\alpha)}{(n-1)^2} x(1-x)\right)+3 \rho (n-1)x+2\right]\\
&& + \frac{a_1(n)}{(n\rho+2)(n\rho+3)} \left[2{\rho}^2(n-1)x+{\rho}^2+3\rho \right],\\
\textit{L}_{n,\alpha,\rho}^{M,1}(e_0,x) &=& a_1(n)x-2a_0(n),\\
\textit{L}_{n,\alpha,\rho}^{M,1}(e_1,x) &=& (a_1(n)x-2a_0(n))\left[\frac{(n-1)\rho x}{n\rho+2}\right]
- \frac{a_0(n)(\rho+2)}{n\rho+2}+\frac{a_1(n)x(\rho+1)}{n\rho+2},\\
\textit{L}_{n,\alpha,\rho}^{M,1}(e_2,x) &=& \frac{a_1(n)x-2a_0(n)}{(n\rho+2)(n\rho+3)}[{\rho}^2(n-1)^2x^2+{\rho}^2(n-1+2(1-\alpha))x(1-x)+3\rho(n-1)x+2]\\
&& + \frac{a_1(n)x-a_0(n)}{(n\rho+2)(n\rho+3)}[2{\rho}^2(n-1)x+{\rho}^2+3\rho].
\end{eqnarray*}
Since $a_1(n)$ is convergent, assume $ \displaystyle\lim_{n\rightarrow\infty} a_1(n)= l_1,$ therefore we obtain
$$\lim_{n\rightarrow\infty} \textit{K}_{n,\alpha,\rho}^{M,1}(f;x)=l_1(1+x)f(x) \,\,\mbox{uniformly on}\, [0,1],$$
$$\lim_{n\rightarrow\infty} \textit{L}_{n,\alpha,\rho}^{M,1}(f;x)=[l_1(1+x)-1]f(x)\,\,\mbox{uniformly on}\, [0,1].$$
By using both the above limits and equation (\ref{m1}), we get the desired result.
\end{proof}

\begin{thm}\label{t4}
Let $a_i(n), i=0,1$ are convergent sequences satisfying the conditions (\ref{eq1}), (\ref{eq3}), and $l_i=\displaystyle\lim_{n \rightarrow \infty}a_i(n).$ If $f^{\prime\prime} \in C[0,1],$ then:
\begin{eqnarray*}
\lim_{n\rightarrow \infty}n(\textit{J}_{n,\alpha,\rho}^{M,1}(f;x)-f(x))= \frac{(1-2x)((\rho+2)l_0+(\rho+1)l_1)}{\rho}f^\prime(x)+\frac{x(1-x)(1+\rho)(2l_0+l_1)}{2\rho}f^{\prime\prime}(x),
\end{eqnarray*}
uniformly on [0,1].
\end{thm}
\begin{proof} With the similar lines as in the proof of Theorem \ref{t2}, after applying our operators $\textit{J}_{n,\alpha,\rho}^{M,1}(.,x)$ to the Taylor's formula, it is enough to prove that
$$\lim_{n\rightarrow\infty}n  \textit{J}_{n,\alpha,\rho}^{M,1}(\theta(t,x)(t-x)^2;x)=0.$$
We can rewrite operators (\ref{a1}) in the following way:
\begin{eqnarray}\label{a9}
\textit{J}_{n,\alpha,\rho}^{M,1}(f,x) = \sum_{k=0}^{n-1}p_{n-1,k,\alpha}(x)\left(a(x,n)\int_{0}^1 \mu_{n,k}^\rho(t) f(t)dt+ a(1-x,n)\int_0^1 \mu_{n,k+1}^\rho(t) f(t)dt\right).
\end{eqnarray}
For $\epsilon>0, \,\exists$ a  $\delta>0$ such that $|t-x|<\delta,$ then $|\theta(t,x)|<\epsilon.$\\
Divide  the interval $[0,1]$ into two parts as below:
\begin{eqnarray*}
I_1=(x-\delta,x+\delta) \cap [0,1],\,\qquad \qquad \qquad \qquad \qquad \qquad I_2=[0,1]\setminus(x-\delta,x+\delta).
\end{eqnarray*}
Since $a_i(n), i=0,1$ are convergent, so are bounded. Thus, $\exists \,\, C>0$ such that $|a(x,n)|<C.$\\
\noindent $|\textit{J}_{n,\alpha,\rho}^{M,1}(\theta(t,x)(t-x)^2);x)|$
\begin{eqnarray*}
 &\leq&  C\sum_{k=0}^{n-1}p_{n-1,k,\alpha}(x)
\left( \int_0^1 \mu_{n,k}^\rho(t) |\theta(t,x)|(t-x)^2 dt+\int_0^1 \mu_{n,k+1}^\rho(t)|\theta(t,x)|(t-x)^2 dt\right)\\
&=& C \sum_{k=0}^{n-1}p_{n-1,k,\alpha}(x) \left(\int_{I_1}\mu_{n,k}^\rho(t) |\theta(t,x)|(t-x)^2 dt +\int_{I_2}\mu_{n,k}^\rho(t) |\theta(t,x)|(t-x)^2 dt\right.\\&&\left. +\int_{I_1} \mu_{n,k+1}^\rho(t) |\theta(t,x)|(t-x)^2 dt +\int_{I_2}\mu_{n,k+1}^\rho(t) |\theta(t,x)|(t-x)^2 dtdt\right)\\
&<& C\sum_{k=0}^{n-1} p_{n-1,k,\alpha}(x)\left[\epsilon \left(\int_{I_1}\mu_{n,k}^\rho(t)(t-x)^2 dt+\int_{I_1} \mu_{n,k+1}^\rho(t)(t-x)^2 dt\right)\right.\\
&&\left.+ \frac{M}{{\delta}^2}\left(\int_{I_2}\mu_{n,k}^\rho(t)(t-x)^4 dt+\int_{I_2}\mu_{n,k+1}^\rho(t)(t-x)^4 dt\right)\right], \,\, \mbox{where}\,\, M= \sup_{0\leq t\leq 1} |\theta(t,x)|\\
&\leq & \epsilon C\sum_{k=0}^{n-1}p_{n-1,k,\alpha}(x)\left[\int_0^1\mu_{n,k}^{\rho}(t)(t-x)^2dt+\int_0^1\mu_{n,k+1}^{\rho}(t)(t-x)^2dt\right]\\
&& + \frac{M C}{{\delta}^2}\sum_{k=0}^{n-1}p_{n-1,k,\alpha}(x)\left[\int_0^1\mu_{n,k}^{\rho}(t)(t-x)^4dt+\int_0^1\mu_{n,k+1}^{\rho}(t)(t-x)^4dt\right]\\
&\leq& \frac{\epsilon C}{(n\rho+2)(n\rho+3)} \left[n(x^2(-2{\rho}^2-2\rho)+x(2{\rho}^2+2\rho))+({\rho}^2+3\rho+4)\right.\\
&&- \left.(12+12\rho+4\alpha{\rho}^2)x(1-x)\right]\\
&&+ \frac{M C}{{\delta}^2(n\rho+2)(n\rho+3)(n\rho+4)(n\rho+5)}\left[6{\rho}^2(1+{\rho})^2x^2(1-x)^2n^2\right.\\
&& +n\left\{4x(1-x)\rho(1+\rho)(13+11\rho+13{\rho}^2-(43+43\rho+(7+6\alpha){\rho}^2)x(1-x))\right\}\\
&& + \left\{ {\rho}^4+10{\rho}^3+35{\rho}^2+50\rho+48-240x(1-x)(1-x(1-x))\right.\\
&& - 80x(1-x)(4-6x(1-x))\rho-120x(1-x){\rho}^2\\
&& +80x(1-x)(1-6x(1-x)){\rho}^3+48x(1-x)(1-5x(1-x)){\rho}^4\\
&& -\left.\left.4\alpha x(1-x){\rho}^2(35+45\rho+16{\rho}^2-12(10+15\rho+6{\rho}^2)x(1-x))\right\}\right].
\end{eqnarray*}
Thus, from the last inequality, the proof is completed.
\end{proof}

Next, we prove direct estimates for our operators $\textit{J}_{n,\alpha,\rho}^{M,1}(f,x).$ For this, we need the following definitions:
\begin{definition}{\bf{(first order modulus of continuity)}}
Let $f(x)$ be bounded on $[a,b],$ the modulus of continuity of $f(x)$ on $[a, b],$ is denoted by $\omega(f;\delta),$ and is defined for $\delta>0$ as:
\begin{eqnarray*}
\omega(f;\delta)=\sup_{\substack{(x,y)\in [a,b] \\ |x-y|\leq \delta}} |f(x)-f(y)|,
\end{eqnarray*}
with the following relation for $\lambda>0$
\begin{eqnarray}\label{a8}
\omega(f;\lambda \delta)\leq (1+\lambda)\,\omega(f;\delta).
\end{eqnarray}
\end{definition}

\begin{definition}{\bf{(Lipschitz continuity)}}
A function $f$ satisfies the Lipschitz condition of order $\tau$ with the constant $L$ on $[a, b],$ that is
\begin{eqnarray*}
|f(x_1)-f(x_2)|\leq L\, |x_1-x_2|^\tau,\qquad 0< \tau\leq 1,\quad L>0, \quad \mbox{for}\quad x_1,x_2\in [a,b],
\end{eqnarray*}
iff $\omega(f;\delta)\leq L\, \delta^\tau.$
\end{definition}

\begin{thm}\label{t5}
If $f$ is bounded for $x\in[0,1], \,\, a_0(n)$ is a bounded sequence and $a_i(n), i=0,1$ satisfy the condition (\ref{eq1}) then $||\textit{J}_{n,\alpha,\rho}^{M,1}f-f||\leq (1+3|a_0(n)|) C_3 \,\omega\left(f;\frac{1}{\sqrt n}\right),$ where $||.||$ is the uniform norm over $[0,1], \, \omega(f;\delta)$ is the first order modulus of continuity and $C_3$ is a positive constant.
\end{thm}
\begin{proof} From the definition of our operators and using relation (\ref{a8}), (by taking $\lambda=\sqrt n\, |t-x|, \delta=\frac{1}{\sqrt n}$), we have
\begin{eqnarray*}
|\textit{J}_{n,\alpha,\rho}^{M,1}(f;x)-f(x)| &\leq& |a(x,n)|\sum_{k=0}^{n-1}p_{n-1,k,\alpha}(x)\int_0^1 \mu_{n,k}^\rho(t)|f(t)-f(x)|dt\\
&&+ |a(1-x,n)|\sum_{k=0}^{n-1}p_{n-1,k,\alpha}(x) \int_0^1 \mu_{n,k+1}^\rho(t) |f(t)-f(x)|dt\\
&\leq& |a(x,n)|\sum_{k=0}^{n-1} p_{n-1,k,\alpha}(x)\int_0^1 \mu_{n,k}^\rho(t)\omega(f;|t-x|)dt\\
&&+ |a(1-x,n)|\sum_{k=0}^{n-1}p_{n-1,k,\alpha}(x) \int_0^1 \mu_{n,k+1}^\rho(t) \omega(f;|t-x|)dt\\
&\leq&  |a(x,n)| \omega\left(f;\frac{1}{\sqrt n}\right)\left[1+\sqrt n\sum_{k=0}^{n-1}p_{n-1,k,\alpha}(x)\int_0^1 \mu_{n,k}^\rho(t)|t-x|dt\right]\\
&&+ |a(1-x,n)|\omega\left(f;\frac{1}{\sqrt n}\right)\left[1+\sqrt n\sum_{k=0}^{n-1}p_{n-1,k,\alpha}(x)\int_0^1 \mu_{n,k+1}^\rho(t)|t-x|dt\right].
\end{eqnarray*}
Now, by using Holder's inequality, we get
\begin{eqnarray*}
\sum_{k=0}^{n-1}p_{n-1,k,\alpha}(x)\int_0^1 \mu_{n,k}^\rho(t) |t-x|dt &\leq& \sum_{k=0}^{n-1}p_{n-1,k,\alpha}(x)\left[\int_0^1 \mu_{n,k}^\rho(t)dt\right]^{\frac{1}{2}}\left[\int_0^1 \mu_{n,k}^\rho(t)(t-x)^2\right]^{\frac{1}{2}}\\
&=& \sum_{k=0}^{n-1}p_{n-1,k,\alpha}(x) \left[\int_0^1 \mu_{n,k}^\rho(t)(t-x)^2 dt\right]^{\frac{1}{2}}\\
&\leq & \left[\sum_{k=0}^{n-1}p_{n-1,k,\alpha}(x)\right]^{\frac{1}{2}}\left[\sum_{k=0}^{n-1}p_{n-1,k,\alpha}(x)\int_0^1 \mu_{n,k}^\rho (t)(t-x)^2dt\right]^{\frac{1}{2}}\\
&=& \sqrt{\frac{n(-{\rho}^2x^2-\rho x^2+{\rho}^2+3\rho x)+(6x^2+6\rho x^2+{\rho}^2x-3\rho x+2)}{(n\rho+2)(n\rho+3)}}.
\end{eqnarray*}
Therefore,
\begin{eqnarray}\label{a5}
\sqrt{n}\left[\sum_{k=0}^{n-1} p_{n-1,k,\alpha}(x)\int_0^1\mu_{n,k}^\rho(t)|t-x| dt\right]^{\frac{1}{2}}\leq \frac{C_1\sqrt{n}}{\sqrt{n\rho+2}} \leq C_1.
\end{eqnarray}
Similarly,
\begin{eqnarray}\label{a6}
\sqrt{n}\left[\sum_{k=0}^{n-1} p_{n-1,k,\alpha}(x)\int_0^1\mu_{n,k+1}^\rho(t)|t-x| dt\right]^{\frac{1}{2}}\leq \frac{C_2\sqrt{n}}{\sqrt{n\rho+2}} \leq C_2.
\end{eqnarray}
Using the inequalities (\ref{a5}) and (\ref{a6}), we get the following relation:
\begin{eqnarray*}
|\textit{J}_{n,\alpha,\rho}^{M,1}(f;x)-f(x)|\leq \omega\left(f;\frac{1}{\sqrt{n}}\right)\left[|a(x,n)|(1+C_1)+|a(1-x,n)|(1+C_2)\right].
\end{eqnarray*}
From equation (\ref{eq1}), we have  $$|a(x,n)|=|a_1(n)x+a_0(n)|\leq |a_1(n)|+|a_0(n)|=|1-2a_0(n)|+|a_0(n)|\leq 1+3|a_0(n)|$$
and $$|a(1-x,n)|\leq 1+3|a_0(n)|.$$
Therefore, we get
\begin{eqnarray}\label{a7}
|\textit{J}_{n,\alpha,\rho}^{M,1}(f;x)-f(x)|\leq \omega\left(f;\frac{1}{\sqrt{n}}\right)C_3 (1+3|a_0(n)|).
\end{eqnarray}
 Thus, proof is completed.
\end{proof}

\begin{corollary}
(i) If we assume $f\in C[0,1]$ in Theorem \ref{t5}, then $\displaystyle \lim_{n\rightarrow\infty}\omega\left(f;\frac{1}{\sqrt{n}}\right)=0,$ which gives another proof of the Theorems \ref{t1} and \ref{t3}.\\
(ii) If $f$ satisfies the Lipschitz condition of order $\tau$ with constant $L$ on $[0,1],$ then result obtained in Theorem \ref{t5} reduces to
\begin{eqnarray*}
|\textit{J}_{n,\alpha,\rho}^{M,1}(f;x)-f(x)|\leq L. C_3\, (1+3|a_0(n)|)\, n^{-\tau/2},
\end{eqnarray*}
$C_3$ is same as defined in (\ref{a7}).
\end{corollary}

Now, we find the errors in terms of modulus of continuity in asymptotic formula for our operators (\ref{a1}) for which $\textit{J}_{n,\alpha,\rho}^{M,1}(e_i;x)=e_i, i=0,1.$ Thus, we get the following conditions:
$$2a_0(n)+a_1(n)=1, a_0(n)(\rho+2)+a_1(n)(\rho+1)=0.$$
By solving these equations, we have $a_0(n)=\dfrac{\rho+1}{\rho}$ and $a_1(n)=-\dfrac{\rho+2}{\rho}.$

\begin{thm}\label{t6}
Let $f\in C^2[0,1], x\in [0,1]$ is fixed. Then
$$|\textit{J}_{n,\alpha,\rho}^{M,1}(f;x)-f(x)-\frac{1}{2} \textit{J}_{n,\alpha,\rho}^{M,1}((t-x)^2;x)f^{\prime\prime}(x)| \leq C\frac{1}{n}\omega \left(f^{\prime\prime},\frac{1}{n}\right)$$
where $C$ is a positive constant independent of $n, x.$
\end{thm}
\begin{proof} Let $f\in C^2[0,1].$ Applying the operators $\textit{J}_{n,\alpha,\rho}^{M,1}(.;x)$ on Taylor's formula, we get:\\
$$|\textit{J}_{n,\alpha,\rho}^{M,1}(f;x)-f(x)-\frac{1}{2}\textit{J}_{n,\alpha,\rho}^{M,1}((t-x)^2;x)f^{\prime\prime}(x)| =\frac{1}{2}|\textit{J}_{n,\alpha,\rho}^{M,1}(\theta(t,x)(t-x)^2;x)|,$$
where $\theta(t,x)=f^{\prime\prime}(\xi_x)-f^{\prime\prime}(x)$ and $\xi_x$ lies between  $t$ and $x.$ By using the relation (\ref{a8}), by taking ($\lambda=\sqrt n\, |t-x|, \delta=\frac{1}{\sqrt n}$), we get
$$|\theta(t,x)|=|f^{\prime\prime}(\xi_x)-f^{\prime\prime}(x)| \leq \omega(f^{\prime\prime};|t-x|) \leq (1+\sqrt n\,|t-x|)\, \omega\left(f^{\prime\prime},\frac{1}{\sqrt n}\right).$$
Also, $|a(x,n)|=|a_1(n)x+a_0(n)|\leq \dfrac{\rho+1}{\rho},$ and $|a(1-x,n)|\leq \dfrac{\rho+1}{\rho},$ for $x\in [0,1]$ and using the definition (\ref{a9}), we get
\begin{eqnarray}\label{eq4}
|\textit{J}_{n,\alpha,\rho}^{M,1}(\theta(t,x)(t-x)^2;x)|\leq \frac{\rho+1}{\rho}\omega\left(f^{\prime\prime},\frac{1}{\sqrt n}\right)(J_1+\sqrt n J_2),
\end{eqnarray}
where
\begin{eqnarray*}
J_1 &=& \sum_{k=0}^{n-1} p_{n-1,k,\alpha}(x)\left[ \int_0^1 \mu_{n,k}^\rho(t)(t-x)^2 dt+\int_0^1 \mu_{n,k+1}^\rho(t)(t-x)^2 dt\right]\\
&=& \frac{1}{(n\rho+2)(n\rho+3)}\left[n(-2{\rho}^2x^2-2\rho x^2+2{\rho}^2x+4\rho x)\right.\\
&& +\left.(12x^2+12\rho x^2-12\rho x+{\rho}^2+3\rho+4-4\alpha {\rho}^2x(1-x))\right] \leq  \frac{(\rho+2)^2(n+2)}{(n\rho+2)(n\rho+3)}.\\
J_2 &=&\sum_{k=0}^{n-1}p_{n-1,k,\alpha}(x)\left[\int_0^1 \mu_{n,k}^\rho(t)|t-x|(t-x)^2 dt+\int_0^1\mu_{n,k+1}^\rho(t)|t-x|(t-x)^2 dt\right]\\
&\leq & \sum_{k=0}^{n-1}p_{n-1,k,\alpha}(x)\left[\int_0^1 \mu_{n,k}^\rho(t)(t-x)^2 dt\right]^{\frac{1}{2}}\left[\int_0^1 \mu_{n,k}^\rho(t)(t-x)^4 dt\right]^{\frac{1}{2}}\\
&&+\sum_{k=0}^{n-1}p_{n-1,k,\alpha}(x)\left[\int_0^1 \mu_{n,k+1}^\rho(t)(t-x)^2 dt\right]^{\frac{1}{2}}\left[\int_0^1 \mu_{n,k+1}^\rho(t)(t-x)^4 dt\right]^{\frac{1}{2}}\\
& \leq & \left[\sum_{k=0}^{n-1} p_{n-1,k,\alpha}(x)\int_0^1 \mu_{n,k}^\rho(t)(t-x)^2 dt\right]^\frac{1}{2}\times\left[\sum_{k=0}^{n-1} p_{n-1,k,\alpha}(x)\int_0^1 \mu_{n,k}^\rho(t)(t-x)^4 dt\right]^\frac{1}{2}\\
&& + \left[\sum_{k=0}^{n-1} p_{n-1,k,\alpha}(x)\int_0^1 \mu_{n,k+1}^\rho(t)(t-x)^2 dt\right]^\frac{1}{2}\times\left[\sum_{k=0}^{n-1} p_{n-1,k,\alpha}(x)\int_0^1 \mu_{n,k+1}^\rho(t)(t-x)^4 dt\right]^\frac{1}{2}.\\
\end{eqnarray*}
Using the results of lemma \ref{l2}, we have:
\begin{eqnarray*}
J_2 & \leq & \frac{P}{n\sqrt{n}}, \,\, \mbox{where}\,\, P \,\, \mbox{is independent of}\,\, n.
\end{eqnarray*}
Now replacing the values of $J_1$ and $J_2$ in the relation (\ref{eq4}), we can find a positive constant $C$ which is independent of $n$ and $x$ such that
\begin{eqnarray}
|\textit{J}_{n,\alpha,\rho}^{M,1}(\theta(t,x)(t-x)^2;x)| \leq C\frac{1}{n} \omega \left(f^{\prime\prime},\frac{1}{\sqrt n}\right).
\end{eqnarray}
Hence, we get the required result.
\end{proof}

Now, we find error estimation of the operators $\textit{J}_{n,\alpha,\rho}^{M,1}(f;x)$ in terms of second order modulus of continuity which gives better results than the results found in Theorem \ref{t5} in terms of second order modulus of continuity.
\begin{thm}
If $f\in C[0,1],$ $a_0(n)=\dfrac{\rho+1}{\rho},a_1(n)=-\dfrac{\rho+2}{\rho},$ then
$$||\textit{J}_{n,\alpha,\rho}^{M,1}(f,.)-f|| \leq C\omega_2\left(f;\frac{1}{n}\right).$$
\end{thm}
\begin{proof}
As we know
\begin{eqnarray*}
||\textit{J}_{n,\alpha,\rho}^{M,1}f|| \leq (|a_0(n)|+|a_1(n)|)||f||.
\end{eqnarray*}
Let $g\in C^2[0,1],$ then from Theorem \ref{t6}, we obtain
\begin{eqnarray*}
|\textit{J}_{n,\alpha,\rho}^{M,1}(g;x)-g(x)| \leq \frac{1}{2}\textit{J}_{n,\alpha,\rho}^{M,1}((t-x)^2;x)|g^{\prime\prime}(x)|+ \frac{C_1}{n}\omega\left(g^{\prime\prime},\frac{1}{\sqrt n}\right),
\end{eqnarray*}
where $C_1$ is a positive constant independent of $n$ and $x.$\\
Using the property of modulus of continuity i.e. $\omega(h,\delta)\leq 2\|h\|,$ we have
\begin{eqnarray*}
||\textit{J}_{n,\alpha,\rho}^{M,1}(g,.)-g||\leq  \frac{C_2}{n}||g^{\prime\prime}||,
\end{eqnarray*}
where $C_2$ is a positive constant independent of $n$ and $x.$
Thus, for $f\in C[0,1],$ we have
\begin{eqnarray}\label{a10}
||\textit{J}_{n,\alpha,\rho}^{M,1}(f,.)-f|| &\leq& ||\textit{J}_{n,\alpha,\rho}^{M,1}(f-g)-(f-g)||+||\textit{J}_{n,\alpha,\rho}^{M,1}(g)-g||\nonumber\\
&\leq& C_3||f-g||+\frac{C_2}{n}||g^{\prime\prime}|| \leq C\left[||f-g||+\frac{1}{n}||g^{\prime\prime}||\right],
\end{eqnarray}
where $C$ and $C_3$ are some positive constants independent of $n$ and $x.$\\
Now, keeping in mind the equivalence of second order modulus of continuity $\omega_2(f,t)$ and $K-$functional
$K_2(f,t^2):=\displaystyle\inf_{g\in C^2[0,1]}\{||f-g||+t^2||g^{\prime\prime}||\}$
i.e. $K_2(f,t^2) \leq \dfrac{7}{2}\omega_2 (f,t), 0\leq t \leq1, f\in C[0,1],$ (see\,[\cite{GON}, Corollary 2.7]),
and then by taking the infimum over all $g\in C^2[0,1],$ to (\ref{a10}) we get the required result.
\end{proof}

\section{$\alpha-$Bernstein P\u{a}lt\u{a}nea operators of second order}
In a similar way, we can define the second order modifications of the operators $\textit{Q}_{n,\rho}^{\alpha}(f,x)$ which is given by:
\begin{eqnarray}\label{a2}
\textit{J}_{n,\alpha,\rho}^{M,2}(f,x)=\sum_{k=0}^n p_{n,k,\alpha}^{M,2}(x)\int_0^1\mu_{n,k}^\rho(t)f(t)dt,
\end{eqnarray}

\begin{eqnarray*}
\mbox{where}\,\,\qquad p_{n,k,\alpha}^{M,2}(x) &=& a(x,n)p_{n-2,k,\alpha}(x)+b(x,n)p_{n-2,k-1,\alpha}(x)+a(1-x,n)p_{n-2,k-2,\alpha}(x),\\
\mu_{n,k}^\rho(t) &=& \frac{t^{k\rho}(1-t)^{(n-k)\rho}}{B(k\rho+1,(n-k)\rho+1)},
\end{eqnarray*}
and $a(x,n)=a_2(n)x^2+a_1(n)x+a_0(n),$ $b(x,n)=b_0(n)x(1-x).$\\
If $a_2(n)=1, a_1(n)=-2, a_0(n)=1, b_0(n)=2,$ then we will get our original operators (\ref{OP}).

\begin{lemma}\label{l3}
The moments of the operators (\ref{a2}) are given by:
\begin{eqnarray*}
\textit{J}_{n,\alpha,\rho}^{M,2}(e_0;x) &=& x^2(2a_2(n)-b_0(n))+x(b_0(n)-2a_2(n))+(2a_0(n)+a_1(n)+a_2(n)),\\
\textit{J}_{n,\alpha,\rho}^{M,2}(e_1;x) &=& \frac{1}{(n\rho+2)}\left[n\{x^3(2a_2(n)-b_0(n))\rho+x^2(-2a_2(n)+b_0(n))\rho+x(2a_0(n)+a_1(n)+a_2(n))\rho\}\right.\\
&& + \left.\left\{x^3(-4a_2(n)+2b_0(n))\rho+x^2(2a_2(n)(3\rho+1)-b_0(n)(3\rho+1))\right.\right.\\
&& + x(-2a_2(n)(3\rho+1)-4\rho a_1(n)-4\rho a_0(n)+b_0(n)(\rho+1))\\
&& +\left.\left.(a_2(n)(1+2\rho)+a_1(n)(1+2\rho)+2a_0(n)(1+\rho))\right\}\right],\\
\textit{J}_{n,\alpha,\rho}^{M,2}(e_2;x) &=& \frac{1}{(n \rho+2)(n \rho+3)}\left[n^2 {\rho}^2 \left\{x^4(2a_2(n)-b_0(n))+x^3 (-2a_2(n)+b_0(n))\right.\right.\\
&& + \left.x^2 (2a_0(n)+a_1(n)+a_2(n)) \right\}+n\left\{-5x^4{\rho}^2(2a_2(n)-b_0(n))\right.\\
&& + x^3\rho((16{\rho}+6)a_2(n)-(7{\rho}+4)b_0(n))\\
&& + \left.x^2\rho(-3(5{\rho}+2) a_2(n)-9{\rho} a_1(n)-10{\rho}a_0(n)+3({\rho}+1)b_0(n))\right.\\
&& +\left.x\rho(6(\rho+1)a_0(n)+(3+5\rho)(a_1(n)+a_2(n)))\right\}+ \left\{2x^4{\rho}^2(2+\alpha)(2a_2(n)-b_0(n))\right.\\
&& -2x^3\rho(2\alpha \rho+4\rho+3)(2a_2(n)-b_0(n))\\
&& + x^2\left({\rho}^2(6(4+\alpha)a_2(n)+2(6+\alpha)a_1(n)+4(2+\alpha)a_0(n)-(5+2\alpha)b_0(n))\right.\\
&& +\left.(9\rho+2)(2a_2(n)-b_0(n))\right)+x\left(-2(2a_2(n)-b_0(n))-3\rho(6a_2(n)+4(a_1(n)+a_0(n))-b_0(n))\right.\\
&& -\left.{\rho}^2(2(8+\alpha)a_2(n)+2(6+\alpha)a_1(n)+4(2+\alpha)a_0(n)-b_0(n))\right)\\
&& +\left.\left.4{\rho}^2(a_2(n)+a_1(n)+a_0(n))+6\rho(a_2(n)+a_1(n)+a_0(n))+2(a_2(n)+a_1(n)+2a_0(n))\right\}\right].
\end{eqnarray*}
To study the uniform convergence of these operators, we take $\textit{J}_{n,\alpha,\rho}^{M,2}(e_0,x)=1,$ which give the following conditions:
$$2a_2(n)-b_0(n)=0,\quad 2a_0(n)+a_1(n)+a_2(n)=1.$$
With both of these conditions, other moments reduce to:
\begin{eqnarray*}
\textit{J}_{n,\alpha,\rho}^{M,2}(e_1,x) &=& x+\frac{1}{n\rho+2}[(1+2\rho-2\rho a_0(n))-2x(1+2\rho-2a_0(n)\rho)],\\
\textit{J}_{n,\alpha,\rho}^{M,2}(e_2,x) &=& x^2-\frac{1}{(n\rho+2)(n\rho+3)}\left[(2a_2(n){\rho}^2+2\alpha{\rho}^2-n{\rho}^2-4{\rho}^2-n\rho-8\rho-6)x(1-x)+(1+2\rho)\right].
\end{eqnarray*}
\end{lemma}
In order to have $\displaystyle\lim_{n\rightarrow \infty}\textit{J}_{n,\alpha,\rho}^{M,2}(e_i;x)=x^i, i=0,1,2,$ we choose undetermined coefficients as:
$$a_0(n)=\frac{1+2\rho}{2\rho}, a_2(n)=\frac{n(\rho+1)}{2\rho},\, b_0(n)=\frac{n(\rho+1)}{\rho}, a_1(n)=\frac{-(n+2)(\rho+1)}{2\rho}.$$
Thus, our operators become:
\begin{eqnarray}{\label{a3}}
\overline{\textit{J}}_{n,\alpha,\rho}^{M,2}(f;x)=\sum_{k=0}^{n}\overline{p}_{n,k,\alpha}^{M,2}(x)\int_0^1\mu_{n,k}^\rho(t)f(t)dt,
\end{eqnarray}
where
\begin{eqnarray*}
\overline{p}_{n,k,\alpha}^{M,2}(x) &=& \left( \frac{n(\rho+1)}{2 \rho} x^2-\frac{(n+2)(\rho+1)}{2 \rho} x+ \frac{1+2\rho}{2 \rho}\right) p_{n-2,k,\alpha}(x)+ \frac{n(\rho+1)}{\rho}x(1-x) p_{n-2,k-1,\alpha}(x){\nonumber}\\
&& + \left(\frac{n(\rho+1)}{2\rho} x^2-\frac{(n-2)(\rho+1)}{2\rho} x-\frac{1}{2\rho}\right) p_{n-2,k-2,\alpha}(x).
\end{eqnarray*}
\begin{lemma}
The moments of the operators ({\ref{a3}}) are:
\begin{eqnarray*}
\overline{\textit{J}}_{n,\alpha,\rho}^{M,2}(e_0;x) &=& 1,\\
\overline{\textit{J}}_{n,\alpha,\rho}^{M,2}(e_1;x) &=& x,\\
\overline{\textit{J}}_{n,\alpha,\rho}^{M,2}(e_2;x) &=& x^2+\frac{1}{(n\rho+2)(n\rho+3)} \left[-1-2 \rho +(6+8 \rho+4{\rho}^2-2\alpha {\rho}^2)x(1-x)\right].
\end{eqnarray*}
\end{lemma}
\begin{lemma}\label{l4}
The central moments of the operators ({\ref{a3}}) are given by:
\begin{eqnarray*}
\overline{\textit{J}}_{n,\alpha,\rho}^{M,2}((t-x)^2;x) &=& \frac{1}{(n\rho+2)(n\rho+3)} \left[-1-2 \rho +(6+8 \rho+4{\rho}^2-2\alpha {\rho}^2)x(1-x)\right],\\
\overline{\textit{J}}_{n,\alpha,\rho}^{M,2}((t-x)^3;x) &=& \frac{2x(1-x)(2x-1)\rho(1+\rho)(2+\rho)n}{(n\rho+2)(n\rho+3)(n\rho+4)}+O\left(\frac{1}{n^3}\right),\\
\overline{\textit{J}}_{n,\alpha,\rho}^{M,2}((t-x)^4;x) &=& -\frac{3x^2(1-x)^2{\rho}^2(1+{\rho})^2n^2}{(n\rho+2)(n\rho+3)(n\rho+4)(n\rho+5)}
+O\left(\frac{1}{n^3}\right),\\
\overline{\textit{J}}_{n,\alpha,\rho}^{M,2}((t-x)^5;x) &=&\frac{30x^2(1-x)^2(2x-1){\rho}^2(1+{\rho})^2(2+\rho))n^2}{(n\rho+2)(n\rho+3)(n\rho+4)(n\rho+5)(n\rho+6)}
+O\left(\frac{1}{n^4}\right),\\
\overline{\textit{J}}_{n,\alpha,\rho}^{M,2}((t-x)^6;x) &=& -\frac{30x^3(1-x)^3{\rho}^3(1+\rho)^3n^3}{(n\rho+2)(n\rho+3)(n\rho+4)(n\rho+5)(n\rho+6)(n\rho+7)}
+O\left(\frac{1}{n^4}\right).
\end{eqnarray*}
\end{lemma}

\begin{thm}
If $f\in C^6[0,1]$ and $x\in [0,1],$ then for sufficiently large $n$, we have:
$$\overline{\textit{J}}_{n,\alpha,\rho}^{M,2}(f;x)-f(x)=O\left(\frac{1}{n^2}\right).$$
\end{thm}
\begin{proof}
Applying the operators $\overline{\textit{J}}_{n,\alpha,\rho}^{M,2}(.;x)$ to Taylor's formula, we get:
\begin{eqnarray*}
\overline{\textit{J}}_{n,\alpha,\rho}^{M,2}(f;x)=f(x)+\sum_{k=1}^6 f^{(k)}(x) \overline{\textit{J}}_{n,\alpha,\rho}^{M,2}\left((t-x)^k:x\right)+\overline{\textit{J}}_{n,\alpha,\rho}^{M,2}\left(\theta(t,x)(t-x)^6;x\right),
\end{eqnarray*}
where $\displaystyle\lim_{t \rightarrow x}\theta(t,x)=0.$\\
We can easily see that from Lemma \ref{l4}, it is enough to prove the following:
\begin{eqnarray*}
\overline{\textit{J}}_{n,\alpha,\rho}^{M,2}\left(\theta(t,x)(t-x)^6;x\right)=0.
\end{eqnarray*}
Now,\\
$|\overline{\textit{J}}_{n,\alpha,\rho}^{M,2}\left(\theta(t,x)(t-x)^6;x\right)|$
\begin{eqnarray*}
 &=& \left|\left[\left(\frac{n(\rho+1)}{2\rho} x^2-\frac{(n+2)(\rho+1)}{2\rho} x+\frac{1+2\rho}{2\rho}\right)\sum_{k=0}^{n-2} p_{n-2,k,\alpha}(x)+ \frac{n(\rho+1)}{\rho} x(1-x) \sum_{k=1}^{n-1} p_{n-2,k-1,\alpha}(x)\right.\right.\\
&& +\left.\left.\left(\frac{n(\rho+1)}{2\rho} x^2-\frac{(n-2)(\rho+1)}{2\rho} x-\frac{1}{2\rho}\right)\sum_{k=2}^n p_{n-2,k-2,\alpha}(x)\right]\int_0^1 \mu_{n,k}^{\rho}(t)\theta(t,x)(t-x)^6dt\right|\\
&\leq & \left|\frac{n(\rho+1)}{2\rho} x^2-\frac{(n+2)(\rho+1)}{2\rho} x+\frac{1+2\rho}{2\rho}\right| \sum_{k=0}^{n-2} p_{n-2,k,\alpha}(x)\int_0^1 \mu_{n,k}^{\rho}(t)\theta(t,x)(t-x)^6dt\\
&& +\left|\frac{n(\rho+1)}{\rho}x(1-x)\right| \sum_{k=1}^{n-1} p_{n-2,k-1,\alpha}(x)\int_0^1 \mu_{n,k}^{\rho}(t)\theta(t,x)(t-x)^6dt\\
&& +\left|\frac{n(\rho+1)}{2\rho}x(x-1)+\frac{2(\rho+1)}{2\rho} x-\frac{1}{2\rho}\right|\sum_{k=2}^n p_{n-2,k-2,\alpha}(x)\int_0^1 \mu_{n,k}^{\rho}(t)\theta(t,x)(t-x)^6 dt.
\end{eqnarray*}
Let $M=sup_{t\in [0,1]}|\theta(t,x)|,$ then we have\\

$|\overline{\textit{J}}_{n,\alpha,\rho}^{M,2}\left(\theta(t,x)(t-x)^6;x\right)|$
\begin{eqnarray*}
&<&M
\left[\frac{n(\rho+1)}{8\rho}+\frac{\rho+1}{\rho}+\frac{1+2\rho}{2\rho}\right]\sum_{k=0}^{n-2} p_{n-2,k,\alpha}(x)\int_0^1 \mu_{n,k}^{\rho}(t)(t-x)^6 dt\\&&+M\left[\frac{n(\rho+1)}{4\rho}\right]\sum_{k=1}^{n-1} p_{n-2,k-1,\alpha}(x)\int_0^1 \mu_{n,k}^{\rho}(t)(t-x)^6dt\\
&& + M\left[\frac{n(\rho+1)}{8\rho}+\frac{2(\rho+1)}{2\rho}+\frac{1}{2\rho}\right]\sum_{k=2}^n p_{n-2,k-2,\alpha}(x)\int_0^1 \mu_{n,k}^{\rho}(t-x)^6dt\\
& = &\frac{M n(\rho+1)}{8\rho}\left[\sum_{k=0}^{n-2} p_{n-2,k,\alpha}(x) \int_0^1 \mu_{n,k}^{\rho}(t-x)^6dt\right.\\
&& \left.+2 \sum_{k=1}^{n-1} p_{n-2,k-1,\alpha}(x)\int_0^1 \mu_{n,k}^{\rho}(t-x)^6dt + \sum_{k=2}^n p_{n-2,k-2,\alpha}(x)\int_0^1 \mu_{n,k}^{\rho}(t-x)^6dt\right]\\
&& +\frac{4\rho+3}{2\rho} M\left[\sum_{k=0}^{n-2} p_{n-2,k,\alpha}(x) \int_0^1 \mu_{n,k}^{\rho}(t-x)^6dt+\sum_{k=2}^n p_{n-2,k-2,\alpha}(x) \int_0^1 \mu_{n,k}^{\rho}(t-x)^6dt\right]\\
& = &\frac{Mn(\rho+1)}{8\rho}\left\{\frac{15x^3{\rho}^3(4(1+\rho)^3(1-3x+3x^2)-x^3(4+4\rho(3+3\rho(3+\rho))))n^3}{\displaystyle\prod_{k=2}^7(n\rho+k)} +O\left(\frac{1}{n^4}\right)\right\}
\end{eqnarray*}
\begin{eqnarray*}
 && +\frac{4\rho+3}{2\rho} M \left\{\frac{15x^3{\rho}^3((1+{\rho})^3(1-3x+3x^2)-x^3(1+\rho(3+\rho(3+\rho))))n^3}{\displaystyle\prod_{k=2}^7(n\rho+k)}\right.\\
 && +\left.\frac{15x^3(1-x)^3{\rho}^3(1+{\rho})^3n^3}{\displaystyle\prod_{k=2}^7(n\rho+k)}+O\left(\frac{1}{n^4}\right)\right\}= O\left(\frac{1}{n^2}\right).
\end{eqnarray*}
Hence, the proof is completed.
\end{proof}

\section{$\alpha-$Bernstein P\u{a}lt\u{a}nea operators of third order}
Continuing in the same way as above, we can modified operators to obtain third order approximation operators which is defined as follows:
\begin{eqnarray}\label{a4}
\textit{J}_{n,\alpha,\rho}^{M,3}(f,x)=\sum_{k=0}^n p_{n,k,\alpha}^{M,3}(x)\int_0^1\mu_{n,k}^\rho(t)f(t)dt,
\end{eqnarray}
where
\begin{eqnarray}
 p_{n,k,\alpha}^{M,3}(x) &=& \overline{a}(x,n)p_{n-4,k,\alpha}(x)+\overline{b}(x,n)p_{n-4,k-1,\alpha}(x)+ \overline{d}(x,n)p_{n-4,k-2,\alpha}(x)\nonumber\\
&& +\overline{b}(1-x,n)p_{n-4,k-3,\alpha}(x)+\overline{a}(1-x,n)p_{n-4,k-4,\alpha}(x),
\end{eqnarray}
and
\begin{eqnarray*}
\overline{a}(x,n) &=& \overline{a}_4(n)x^4+\overline{a}_3(n)x^3+\overline{a}_2(n)x^2+\overline{a}_1(n)x+\overline{a}_0(n),\\
\overline{b}(x,n) &=& \overline{b}_4(n) x^4+\overline{b}_3(n)x^3+\overline{b}_2(n)x^2+\overline{b}_1(n)x+\overline{b}_0(n),\\
\overline{d}(x,n) &=& \overline{d}_0(n)x^2 (1-x)^2.
\end{eqnarray*}
Here $\overline{a}_i(n),\overline{b}_i(n), i=0,1,\dotsb,4$ and $\overline{d}_0(n)$ are unknown sequences in $n,$ which can be determined in such a way that the operators (\ref{a4}) reduce to new operators $\tilde{\textit{J}}_{n,\alpha,\rho}^{M,3}(f;x) \, (say)$ with third order approximation. In order to get $\textit{J}_{n,\alpha,\rho}^{M,3}(e_i,x)=e_i, i=0,1,\dotsb,3,$ we find unknown sequences which are given by
\begin{eqnarray*}
\tilde{a}_0(n) &=& \frac{12{\rho}^3+19{\rho}^2+8\rho+1}{12{\rho}^3}, \tilde{a_1}(n)= -\frac{7{\rho}^2+11\rho+4}{12{\rho}^2}n-\frac{(29+5\alpha){\rho}^3+(48+3\alpha){\rho}^2+30\rho+7}{6{\rho}^3},\\
\tilde{a}_2(n) &=& \frac{(1+\rho)^2}{8{\rho}^2}n^2+\frac{17{\rho}^2+29\rho+12}{12{\rho}^2}n+\frac{(41-11\alpha){\rho}^3+(71-9\alpha){\rho}^2+60\rho+18}{6{\rho}^3},\\
\tilde{a}_3(n) &=& -\frac{(1+\rho)^2}{4{\rho}^2}n^2-\frac{5{\rho}^2+9\rho+4}{6{\rho}^2}n-\frac{(3-\alpha){\rho}^3+(7-\alpha){\rho}^2+6\rho+2}{{\rho}^3},\\
\tilde{a}_4(n) &=& \frac{(1+\rho)^2}{8{\rho}2}n^2, \tilde{b_0}(n)=-\frac{12{\rho}^2+7\rho+1}{6{\rho}^3},\\
\tilde{b}_1(n) &=& \frac{(3{\rho}^2+5\rho+2)}{3{\rho}^2}n+\frac{(16-4\alpha){\rho}^3+(37-3\alpha){\rho}^2+27\rho+7}{3{\rho}^3},\\
\tilde{b}_2(n) &=& -\frac{(1+\rho)^2}{2{\rho}^2}n^2-\frac{8{\rho}^2+14\rho+6}{3{\rho}^2}n-\frac{(34-10\alpha){\rho}^3+(71-9\alpha){\rho}^2+57\rho+18}{3{\rho}^3},\\
\tilde{b}_3(n) &=& \frac{(1+\rho)^2}{{\rho}^2}n^2+\frac{5{\rho^2}+9\rho+4}{3\rho}n+\frac{2((3-\alpha){\rho}^3+(7-\alpha){\rho}^2+6\rho+2)}{{\rho}^3},\\
\tilde{b}_4(n) &=&-\frac{(1+\rho)^2}{2{\rho}^2}n^2, \tilde{d_0}(n)=3\frac{(1+\rho)^2}{4{\rho}^2}n^2.
\end{eqnarray*}
\begin{lemma}
The central moments of the operators $\tilde{\textit{J}}_{n,\alpha,\rho}^{M,3}(f;x)$ are given by
\begin{eqnarray*}
\tilde{\textit{J}}_{n,\alpha,\rho}^{M,3}(t-x;x) &=& \tilde{\textit{J}}_{n,\alpha,\rho}^{M,3}((t-x)^2;x)=\tilde{\textit{J}}_{n,\alpha,\rho}^{M,3}((t-x)^3;x)=0,\\
\tilde{\textit{J}}_{n,\alpha,\rho}^{M,3}((t-x)^4;x) &=& \frac{x(1-x)\rho(1+\rho)[11-(58+106\rho)x(1-x)+27\rho+2{\rho}^2(6-(-29+6\alpha)x(1-x))]n}{\prod_{k=2}^5(n\rho+k)}\\
&& + O\left(\frac{1}{n^4}\right),\\
\tilde{\textit{J}}_{n,\alpha,\rho}^{M,3}((t-x)^5;x) &=& \frac{5x^2(1-x)^2(2x-1){\rho}^2(1+{\rho})^2(5+4\rho)n^2}{\prod_{k=2}^6(n\rho+k)}+O\left(\frac{1}{n^4}\right),\\
\tilde{\textit{J}}_{n,\alpha,\rho}^{M,3}((t-x)^6;x) &=& \frac{15x^3(1-x)^3{\rho}^3(1+\rho)^3n^3}{\prod_{k=2}^7(n\rho+k)}+O\left(\frac{1}{n^4}\right),\\
\tilde{\textit{J}}_{n,\alpha,\rho}^{M,3}((t-x)^7;x) &=& \tilde{\textit{J}}_{n,\alpha,\rho}^{M,3}((t-x)^8;x)=O\left(\frac{1}{n^4}\right),\\
\tilde{\textit{J}}_{n,\alpha,\rho}^{M,3}((t-x)^9;x) &=& \tilde{\textit{J}}_{n,\alpha,\rho}^{M,3}((t-x)^{10};x)=O\left(\frac{1}{n^5}\right).
\end{eqnarray*}
\end{lemma}
The asymptotic order of approximation of the operators $\tilde{\textit{J}}_{n,\alpha,\rho}^{M,3}(f;x)$ to $f(x),$ when $n\rightarrow \infty$ is given in the following result:
\begin{thm}
If $f\in C^{10}[0,1]$ and $x\in[0,1],$ then for sufficiently large $n,$ we have
$$\tilde{\textit{J}}_{n,\alpha,\rho}^{M,3}(f;x)-f(x)=O\left(\frac{1}{n^3}\right).$$
\end{thm}

\section{Numerical Results}
In the present section, we give the numerical examples to validate our theoretical results and error estimation by using maple algorithms:
\begin{example}
Let $f(x)=sin(2\pi x)+2sin\left(\frac{1}{2}\pi x\right),\, n=10, \rho=5, \alpha=0.2, a_0(n)=\dfrac{n-1}{2n}$ and $a_1(n)=\dfrac{1}{n}.$ The comparison of convergence of the $\alpha-$Bernstein-P\u{a}lt\u{a}nea operators and its above modifications of orders one, two and three to $f(x)$ is given in Fig. 1.


\begin{center}
\includegraphics[width=0.62\columnwidth]{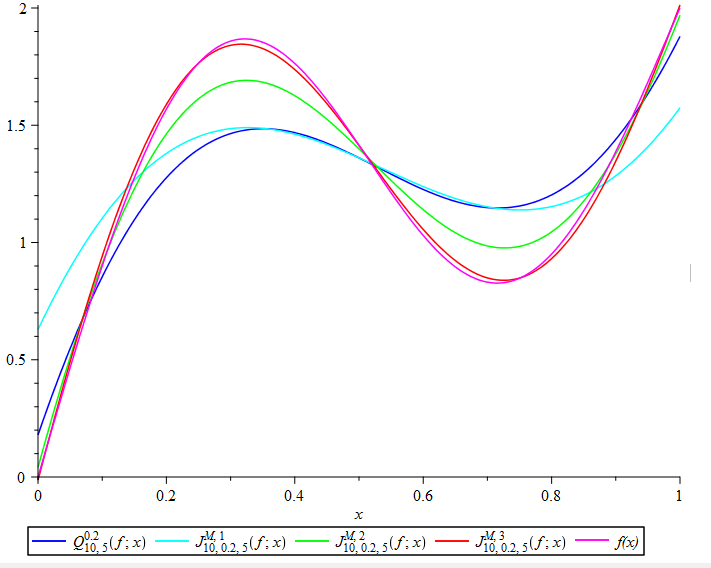}\\
{Figure 1. Approximation process}
\end{center}
Let $E_{n,\rho}^\alpha(f;x)=|f(x)-\textit{Q}_{n,\rho}^\alpha(f;x)|$ and $E_{n,\alpha,\rho}^{M,i}(f;x)=|f(x)-\textit{J}_{n,\alpha,\rho}^{M,i}(f;x)|,\,i=1,2,3$ be error function of $\alpha-$Bernstein P\u{a}lt\u{a}nea operators and its modifications respectively. The error of approximation of these operators are given in figure 2. From both the figures, we can conclude that our modified operators are converging faster than original $\alpha-$Bernstein P\u{a}lt\u{a}nea operators. Also, we have given error of approximation at some certain points in Table 1.
\begin{center}
\includegraphics[width=0.62\columnwidth]{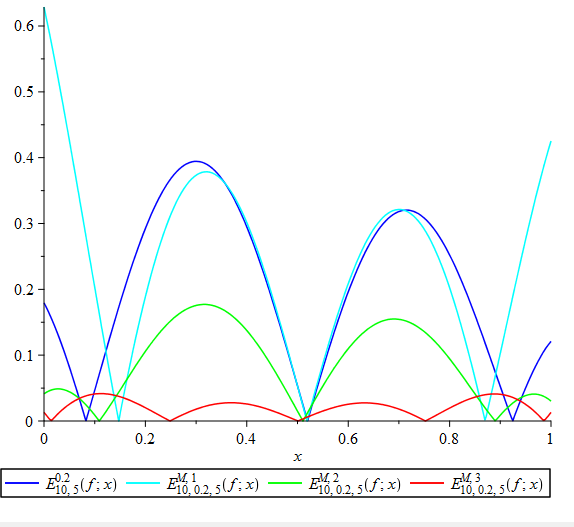}\\
\vskip0.1in
{Figure 2. Error estimation}
\end{center}
\begin{table}[htb]
\centering
\caption{Error of Approximation $E_{n,\rho}^\alpha$ and $E_{n,\alpha,\rho}^{M,i}, i=1,2,3, n=10, \rho=5,\alpha=0.2$}
\begin{tabular}{|c|c|c|c|c|}
\hline $x$ &  ${E}_{10,5}^{0.2}(f;x)$ & $E_{10,0.2,5}^{M,1}(f;x)$ &
$E_{10,0.2,5}^{M,2}(f;x)$ & $E_{10,0.2,5}^{M,3}(f;x)$\\
\hline 0.1 & 0.0465927917 & 0.2029191714 & 0.0090258508 & 0.0410711480\\
\hline 0.2 & 0.292607380 & 0.189033401 & 0.105857578 & 0.020163736\\
\hline 0.3 & 0.394424647 & 0.373376025 & 0.175570300 & 0.018015963\\
\hline 0.4 & 0.294804927 & 0.301231910 &  0.137423887& 0.025812675\\
\hline 0.5 &  0.053236320 & 0.053236321 & 0.014926178  & 0.000239280\\
\hline 0.6 & 0.196044423 & 0.208534327 & 0.109839626 & 0.025493047 \\
\hline 0.7 & 0.3182317026 & 0.3213021846  & 0.1544639662 & 0.0180914623 \\
\hline 0.8 & 0.2521506342 & 0.2023484272 & 0.0944769302  & 0.0196894259\\
\hline 0.9 & 0.052294816 & 0.102854286 & 0.008980412 & 0.040540562\\
\hline
\end{tabular}
\end{table}
\end{example}


\begin{example}
Let us choose $f(x)=x \cos(2\pi x)$, $ \rho=4, \alpha=0.3, a_0(n)=\dfrac{n-1}{2n}$ and $a_1(n)=\dfrac{1}{n}.$ The behaviour of $\alpha-$Bernstein-P\u{a}lt\u{a}nea operators $\textit{Q}_{n,\rho}^{\alpha}(f;x)$ and its three modifications $\textit{J}_{n,\alpha,\rho}^{M,i}(f;x)$ where $i=1,2,3$ to $f(x)$ for $n=10, 15, 20$ is given in Fig. 3, 5, 7. We can observe from figures that our modifications are converging to a function as we increase the value of $n$ and also give better convergence than original operators.

\begin{center}
\includegraphics[width=0.65\columnwidth]{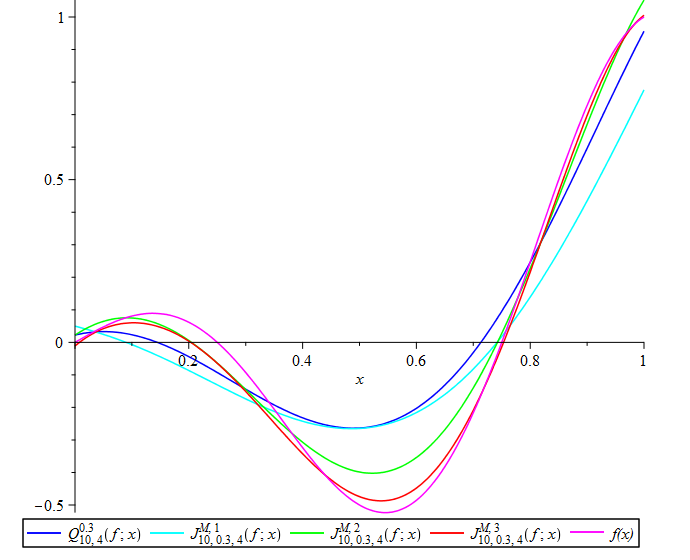}\\
{Figure 3. Approximation process}
\end{center}
The error of approximation of $\alpha-$Bernstein P\u{a}lt\u{a}nea operators and its modifications of order $1,2,3$ are given in Fig. 4,6,8 for $n=10,15,20$ respectively. It can be easily seen that error estimation by our modifications are less than original $\alpha-$Bernstein P\u{a}lt\u{a}nea operators. Also, the error of operators and its modifications at some points are given in Table 2, 3, 4 at the values $n=10,15,20$ respectively.
\begin{center}
\includegraphics[width=0.65\columnwidth]{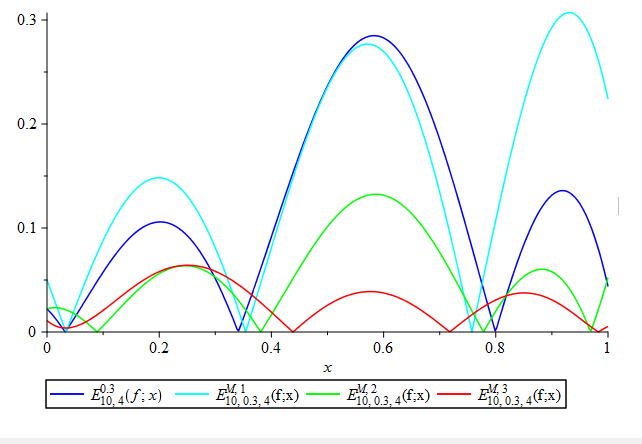}\\
{Figure 4. Error estimation}
\end{center}
\begin{table}[htb]
\centering
\caption{Error of Approximation $E_{n,\rho}^\alpha$ and $E_{n,\alpha,\rho}^{M,i}, i=1,2,3, n=10, \rho=4,\alpha=0.3$}
\begin{tabular}{|c|c|c|c|c|}
\hline $x$ &  ${E}_{10,4}^{0.3}(f;x)$ & $E_{10,0.3,4}^{M,1}(f;x)$ &
$E_{10,0.3,4}^{M,2}(f;x)$ & $E_{10,0.3,4}^{M,3}(f;x)$\\
\hline 0.1 & 0.05755638690 & 0.07198182417 & 0.00593343135 & 0.02211119337\\
\hline 0.2 & 0.1058403144 & 0.1347020007 & 0.05618431396 & 0.00086323684\\
\hline 0.3 & 0.05131069544 & 0.07328503894 & 0.05274231364 & 0.01900023794\\
\hline 0.4 & 0.0916546188 & 0.0844657778 & 0.0166755487 & 0.0079394543\\
\hline 0.5 & 0.2366464156 & 0.2366464156 & 0.1013426174  & 0.0193312761\\
\hline 0.6 & 0.2827508318 & 0.2698804172 & 0.1314792584 &  0.0277669661\\
\hline 0.7 & 0.1851049397  & 0.1344774017  & 0.0765221240 & 0.0019799015 \\
\hline 0.8 & 0.0010362418 & 0.1078472481 & 0.0202148003  & 0.0324610925\\
\hline 0.9 & 0.1316309131 & 0.2930201701 & 0.0579303733 & 0.0309826354\\
\hline
\end{tabular}
\end{table}

\begin{center}
\includegraphics[width=0.59\columnwidth]{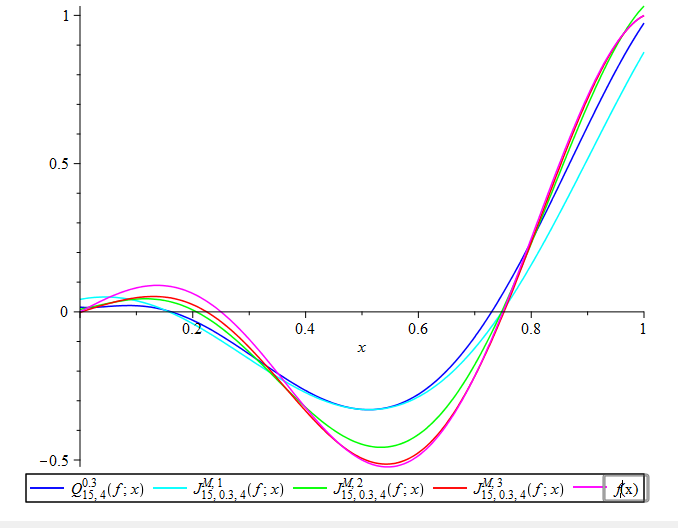}\\
{Figure 5. Approximation process}
\end{center}
\begin{center}
\includegraphics[width=0.59\columnwidth]{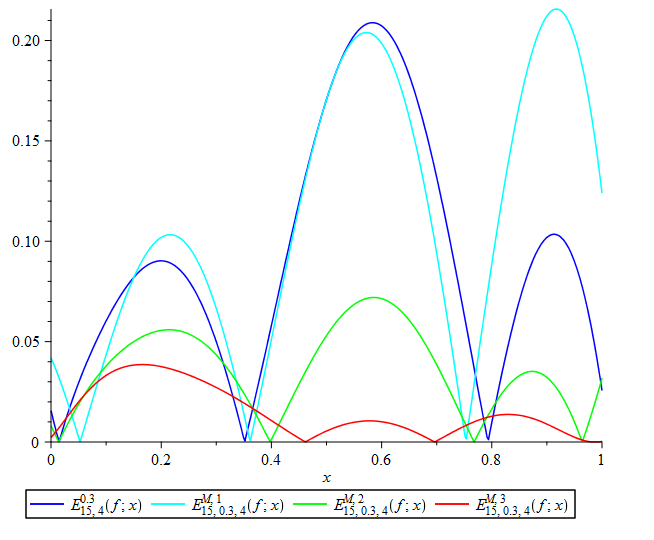}\\
{Figure 6. Error estimation}
\end{center}

\begin{table}[htb]
\centering
\caption{Error of Approximation $E_{n,\rho}^\alpha$ and $E_{n,\alpha,\rho}^{M,i}, i=1,2,3, n=15, \rho=4,\alpha=0.3$}
\begin{tabular}{|c|c|c|c|c|}
\hline $x$ &  ${E}_{15,4}^{0.3}(f;x)$ & $E_{15,0.3,4}^{M,1}(f;x)$ &
$E_{15,0.3,4}^{M,2}(f;x)$ & $E_{15,0.3,4}^{M,3}(f;x)$\\
\hline 0.1 & 0.03861056189  & 0.04342340681 & 0.00028346673 & 0.00994590237\\
\hline 0.2 & 0.08120272655 & 0.1019322964 & 0.02979966536 & 0.00261425244\\
\hline 0.3 & 0.04762100304 & 0.06648839344 & 0.03350486624 & 0.00718895284\\
\hline 0.4 & 0.0584777139  & 0.0510375462 & 0.0037061452 & 0.0043217338 \\
\hline 0.5 & 0.1707169136 & 0.1707169135 & 0.0534421617  & 0.0067597319\\
\hline 0.6 & 0.2073066153 & 0.1986262669 & 0.0714731511 & 0.0102423637 \\
\hline 0.7 & 0.1318061469 &  0.0940588932 & 0.0384911205 &  0.0005598941\\
\hline 0.8 & 0.0097174609 & 0.0887152715 &  0.0165306134 & 0.0128506443\\
\hline 0.9 & 0.1019438820 & 0.2122610177 & 0.0321883935 & 0.0089215656\\
\hline
\end{tabular}
\end{table}

\begin{center}
\includegraphics[width=0.59\columnwidth]{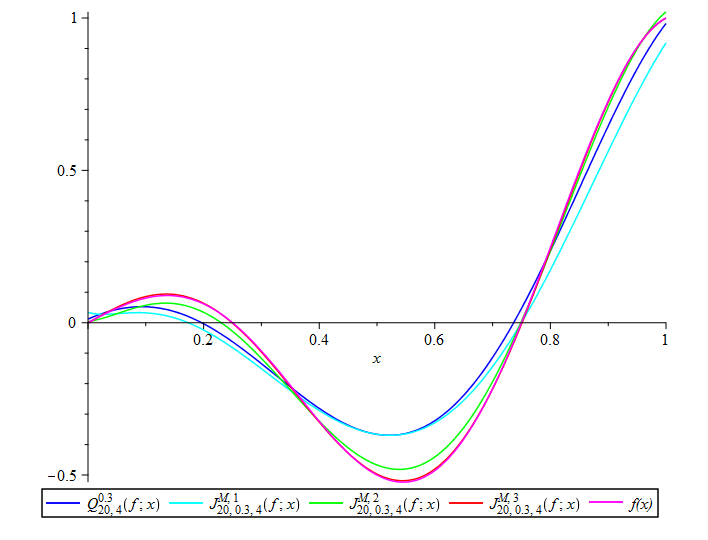}\\
{Figure 7. Approximation process}
\end{center}
\begin{center}
\includegraphics[width=0.59\columnwidth]{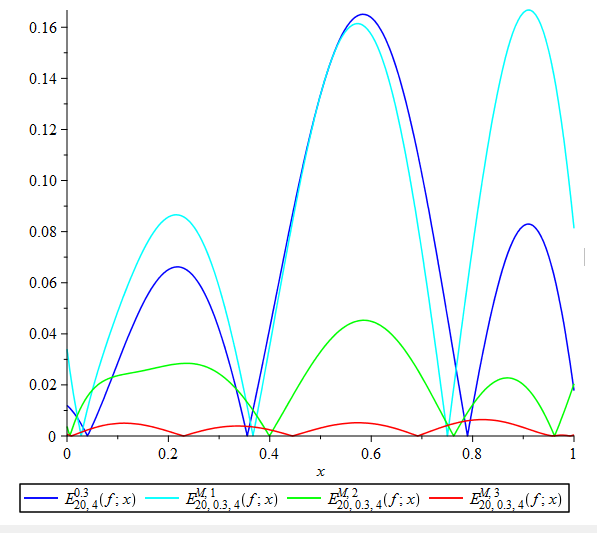}\\
{Figure 8. Error estimation}
\end{center}
\end{example}

\begin{table}[htb]
\centering
\caption{Error of Approximation $E_{n,\rho}^\alpha$ and $E_{n,\alpha,\rho}^{M,i}, i=1,2,3, n=20, \rho=4,\alpha=0.3$}
\begin{tabular}{|c|c|c|c|c|}
\hline $x$ &  ${E}_{20,4}^{0.3}(f;x)$ & $E_{20,0.3,4}^{M,1}(f;x)$ &
$E_{20,0.3,4}^{M,2}(f;x)$ & $E_{20,0.3,4}^{M,3}(f;x)$\\
\hline 0.1 & 0.02861514285  & 0.02978527157 & 0.00135692531 & 0.00493671473\\
\hline 0.2 & 0.06515132595 & 0.08083293178 & 0.01810232856 & 0.00184414784\\
\hline 0.3 & 0.04171602804 & 0.05754420964 & 0.02246595754 & 0.00339046874\\
\hline 0.4 & 0.0423780236 & 0.0356126371 & 0.0005113770 & 0.0023993240\\
\hline 0.5 & 0.1336280335 & 0.1336280335 & 0.0330134666  & 0.003147654\\
\hline 0.6 & 0.1637928626 & 0.1572063097 & 0.0448962830 & 0.0048952830 \\
\hline 0.7 & 0.1024392756 & 0.0721778840  & 0.0230362411 & 0.0006129461 \\
\hline 0.8 & 0.0113308695 & 0.0741312889 & 0.0121952799  & 0.0062328086\\
\hline 0.9 & 0.0822517165 & 0.1657103520 & 0.0201291953 & 0.0033482141\\
\hline
\end{tabular}
\end{table}

{\bf{Acknowledgements:}} The first author is thankful to the "University Grants Commission (UGC)" India for financial support to carry out the above research work.


\end{document}